\documentclass[11pt,a4paper]{article}
\usepackage[top=3.5cm, bottom=3.5cm, left=3cm, right=3cm]{geometry}
\usepackage[latin1]{inputenc}
\usepackage{csquotes}
\usepackage{amsmath}
\usepackage{amsthm}
\usepackage{amsfonts}
\usepackage{amssymb}
\usepackage{graphics}
\usepackage{float}
\usepackage[hang,flushmargin]{footmisc} 

\usepackage{amsmath,amsthm,amssymb,graphicx, multicol, array}
\usepackage{enumerate}
\usepackage{enumitem}
\usepackage{xcolor}

\usepackage[pagebackref,bookmarks,colorlinks,breaklinks,linktoc=page]{hyperref}
\hypersetup{linkcolor=blue,citecolor=red,filecolor=blue,urlcolor=blue} 

\usepackage{tabto}

\numberwithin{equation}{section}

\usepackage{tikz}
\usepackage{pgfplots}

\usetikzlibrary{matrix}

\usepackage{mathrsfs}

\DeclareMathOperator{\li}{li}

\newtheorem{thm}{Theorem}[section]
\newtheorem{lem}{Lemma}[section]
\newtheorem{theorem}{Theorem}[section]

\usepackage{fancyhdr}
\usepackage{tocloft}

\title{Legendre Conjecture over Arithmetic Progressions}
\date{}
\author{N. A. Carella}

\begin{document}
\maketitle

\textbf{\textit{Abstract}:} 
Let $1\leq a<q$ be a pair of small integers such that $\gcd(a,q)=1$ and let $x>1$ be a large number. This note discusses the existence of a short sequence of primes $p\equiv a\bmod q$ between two squares $x^2$ and $(x+1)^2$. 
\let\thefootnote\relax\footnote{ \today \date{} 
}


\section{Introduction }\label{S3377A}
A recent result on the theory of primes in short intervals claims that each short interval $[n^2,(n+1)^2]$ contains the expected number of primes $\pi((n+1)^2)-\pi(n^2)=(1+o(1))x/\log x$ with at most $O(x^{1/5+\varepsilon})$ exceptional values $n\leq x$, where  $x$ is a large number and $\varepsilon>0$, see \cite[Theorem 1.1]{BD2011}. This result proves the Legendre conjecture on the existence of primes between two squares for almost all integers $n\geq1$.\\

This note provides a partial result for an extended form of the Legendre conjecture on arithmetic progression, this contribution is proved in \hyperlink{thm3377D.050}{Theorem} \ref{thm3377D.050} unconditionally. The preliminary results in the section provides explicit inequalities for the prime number theorem in arithmetic progressions.

\section{Primes in Arithmetic Progressions} \label{S3377D} \hypertarget{S3377D}
Let $x\geq1$ be a large number and let $1\leq a<q$ be a pair of relatively prime integers. The prime counting function over arithmetic progressions and the associated weighted prime counting function are defined by
\begin{equation}\label{eq3377D.100b}
	\pi(x,q,a)=\sum_{\substack{p\leq x\\p\equiv a\bmod q}}1. 
\end{equation}
The associated weighted prime counting functions are defined by
\begin{equation}\label{eq3377D.100d}
	\psi(x,q,a)=\sum_{\substack{p^k\leq x\\p\equiv a\bmod q}}\log p\quad \text{ and } \quad\theta(x,q,a)=\sum_{\substack{p\leq x\\p\equiv a\bmod q}}\log p, 
\end{equation}
where $k\geq0$. The logarithm integral is the defined by the expression 
\begin{equation}\label{eq3377D.100f}
	\li(x)=\int_2^x\frac{1}{\log t}\;dt. 
\end{equation}{\tiny }
\begin{theorem} \label{thm3377D.100}\hypertarget{thm3377D.100} Let $a$ and $q$ be integers with $1\leq q \leq 10^5$ and $\gcd(a, q) = 1$. If $x \geq 10^3$, then
	\begin{enumerate}[font=\normalfont, label=(\roman*)]
		\item $\displaystyle \left| \psi(x,q,a)-\frac{x}{\varphi(q)}\right|<0.19  \frac{x}{\log x},$\\[.1cm]
		\item $\displaystyle \left| \theta(x,q,a)-\frac{x}{\varphi(q)}\right|<0.40  \frac{x}{\log x},$\\[.1cm]
		\item $\displaystyle \left| \pi(x,q,a)-\frac{\li(x)}{\varphi(q)}\right|<0.53  \frac{x}{(\log x)^2}.$\\[.1cm]
	\end{enumerate}
\end{theorem}	
\begin{proof}[\textbf{Proof}] This is Corollary 1.7 in \cite{BR2018}. Other cases of the theorem for other ranges of $q$ are also derived in this paper.
\end{proof}

Earlier and similar explicit estimates for the prime counting functions and related topics were developed in \cite[Theorem 5.2]{DP2016}, et alii. Lastly, there are the conditional explicit estimates as in \cite[Theorem 10]{SL1976}.\\

The analysis of the oscillations of the differences 
\begin{equation}\label{eq3377D.150c}
	\psi(x)-x  \quad \text{ and }\quad  \psi(x+z)-\psi(x)
\end{equation}
and other associated differences are difficult analysis and are topics of current research, there is a vast literature on these topic, confer \cite[p.\; 191]{EE1985}, \cite[Theorem 15.11]{MV2007}, \cite[p.\; 306]{IA1985}, \cite{RH1987}, \cite{BH2000}, \cite{KW2009}, \cite{SD2010} and 
other sources for details.  However, the analysis of the differences

\begin{equation}\label{eq3377D.150d}
(x+z)-x \quad \text{ and }\quad  \frac{x+z}{\log(x+z)}-\frac{x}{\log x}
\end{equation} are much simpler since these are strictly monotonic increasing functions. The explicit formula in \hyperlink{thm3377D.100}{Theorem} \ref{thm3377D.100} facilitate the derivations of both lower bound and upper bound of the complicated difference $\psi(x+z)-\psi(x)$ in terms of the simpler strictly monotonic functions as demonstrated here.
\begin{lem} \label{lem3377D.150A} \hypertarget{lem3377D.150A} Let $a$ and $q$ be integers with $1\leq q \leq 10^5$ and $\gcd(a, q) = 1$, and let $\varepsilon>0$ be a small number. If $x$ is a large number and $x^{\varepsilon}\leq z\leq x$, then
	\begin{enumerate}[font=\normalfont, label=(\roman*)]
	\item $\displaystyle \psi(x+z,q,a)-\psi(x,q,a)\geq  \frac{x+z}{\varphi(q)}-\frac{x}{\varphi(q)} -(0.19)\left( \frac{x+z}{\log (x+z)}-\frac{x}{\log x}\right),$\\[.1cm]
		\item $\displaystyle \psi(x+z,q,a)-\psi(x,q,a)\leq  \frac{x+z}{\varphi(q)}-\frac{x}{\varphi(q)} +(0.19)\left( \frac{x+z}{\log (x+z)}-\frac{x}{\log x}\right).$
\end{enumerate}
\end{lem}

\begin{proof}[\textbf{Proof}] The functions $f(x)=x/\varphi(q)\geq0$ and $g(x)=x/\log x\geq0$ are strictly monotonically increasing for all sufficiently large real numbers $x\geq x_0(q)$. Moreover, the sums and differences
	\begin{equation}\label{eq3377D.150h}
		f(x+z)-f(x)>0, \quad g(x+z)-g(x)>0, \quad f(x)\pm g(x)>0
	\end{equation}		
	are strictly monotonically increasing nonnegative quantities for all sufficiently large real numbers $x\geq x_0(q)$. The claims in \eqref{eq3377D.150h} are routine exercises. \\
	
	Taking the difference of the inequalities
	
	\begin{equation}\label{eq3377D.150i}
		\frac{x+z}{\varphi(q)}-0.19  \frac{x+z}{\log (x+z)}\leq\psi(x+z,q,a)\leq\frac{x+z}{\varphi(q)}+0.19  \frac{x+z}{\log (x+z)},\nonumber
	\end{equation}	
	and	
	\begin{equation}\label{eq3377D.150j}
		\frac{x}{\varphi(q)}-0.19  \frac{x}{\log x}\leq\psi(x,q,a)\leq\frac{x}{\varphi(q)}+0.19  \frac{x}{\log x},
	\end{equation}		
see \hyperlink{thm3377D.100}{Theorem} \ref{thm3377D.100} (i), yields 	
	\begin{eqnarray}\label{eq3377D.150r}
		&& \left(  \frac{x+z}{\varphi(q)} -(0.19)\frac{x+z}{\log (x+z)}\right)-\left( 	\frac{x}{\varphi(q)}-	(0.19)\frac{x}{\log x}\right)\nonumber\\[.3cm]	&\leq&\psi(x+z,q,a)-\psi(x,q,a)\\[.3cm]
		&\leq&\left(  \frac{x+z}{\varphi(q)} +(0.19)\frac{x+z}{\log (x+z)}\right)-\left( 	\frac{x}{\varphi(q)}+	(0.19)\frac{x}{\log x}\right)    \nonumber.
	\end{eqnarray}	
Rearranging and separating \eqref{eq3377D.150r} yield the claimed inequalities. Lastly, observe that the sums and differences in \eqref{eq3377D.150h} are well defined and nontrivial nonnegative real numbers for all $x^{\varepsilon}\leq z\leq x$ as $x\to\infty$.
\end{proof}

The smallest short interval $[x,x+z]$ possible is a topic of a long line of research, there are many conjectures and partial results. According to one of the earliest result, attributed to Westzynthius, almost every short interval $[x,x+z]$ with
\begin{equation}\label{eq3377D.150s}
z\ll \frac{\log x\log\log \log x}{\log \log\log \log x}
\end{equation}
is primefree ---  short intervals containing twin primes and prime $k$-tuples, where $k\geq2$ are small integers, are the exceptions. In contrast, almost every short interval $[x,x+z]$ with
\begin{equation}\label{eq3377D.150t}
	z\gg (\log x)^{1+\varepsilon},
\end{equation}
and $\varepsilon>0$, contains primes or is conjectured to contain primes, confer \cite{SA1943}, \cite{MS2007}, \cite{GL2023}, \cite{MT2024}, et alia for early results, surveys of the literature, numerical study and other recent developments on this topic. 

\section{Legendre Conjecture in Arithmetic Progressions }\label{S3377V}\hypertarget{S3377V}
The key to estimating the number of primes between two squares is the \textit{symmetric} explicit bounds of the prime counting function. Very recently proved symmetric explicit estimates are stated in \hyperlink{thm3377D.100}{Theorem} \ref{thm3377D.100}. The result in \hyperlink{thm3377D.050}{Theorem} \ref{thm3377D.050} is a simple application of this result. 

\begin{thm} \label{thm3377D.050}\hypertarget{thm3377D.050} Let $1\leq a<q\leq 10^5$, with $\gcd(a,q)=1$ be a pair of integers and let $n\geq 1$ be a large integer. Then, there exists a prime number $p\equiv a\bmod q$ in the range $n^2< p<(n+1)^2$. Furthermore, for a large number $x$, the number of weighted primes over the short interval $[x^2,(x+1)^2]$ has the lower bound
	$$	\psi((x+1)^2,q,a)-\psi(x^2,q,a)\geq \frac{2x}{\varphi(q)}+O\left( \frac{x}{\log x}\right).$$
\end{thm}

\begin{proof}[\textbf{Proof}] Let $X+Z=(x+1)^2$. By \hyperlink{lem3377D.150A}{Lemma} \ref{lem3377D.150A} it follows that there is an effective lower bound for the number of weighted primes over the short interval $[x^2,(x+1)^2]$ can be computed as follows. 
\begin{eqnarray}\label{eq3377D.200u}
			D(x)&=&	\psi((x+1)^2,q,a)-\psi(x^2,q,a)\\[.3cm]
&\geq&\left( \frac{(x+1)^2}{\varphi(q)}-	\frac{x^2}{\varphi(q)}   \right)               -0.19\left(	\frac{(x+1)^2}{\log (x+1)^2}
-\frac{x^2}{\log x^2} \right)\nonumber\\[.3cm]		
			&=&\frac{2x+1}{\varphi(q)}+E(x,q,a).\nonumber
			\end{eqnarray}
			
The error term $E(x,q,a)$ is estimated in \hyperlink{lem3377D.200S}{Lemma} \ref{lem3377D.200S}. Summing these estimates produces the lower bound
\begin{eqnarray}\label{eq3377D.200d}
\psi((x+1)^2,q,a)-\psi(x^2,q,a)
&\geq &	\frac{2x+1}{\varphi(q)}-0.19 \left( \frac{2x+1}{2\log x}+O\left( \frac{x}{(\log x)^2}\right)\right)\nonumber\\[.3cm]
&\geq &	\frac{2x}{\varphi(q)}+O\left( \frac{x}{\log x}\right)	\nonumber
\end{eqnarray}
for all large $x$ and small parameter $q=o(\log x)$. In particular, this is valid for $1\leq a<q\leq 10^5$, with $\gcd(a,q)=1$. Therefore, the quantity $D(x)$ is unbounded as $x\to\infty$.
\end{proof}


\begin{lem} \label{lem3377D.200S} \hypertarget{lem3377D.200S}If $x$ is a large number then 
	$$|E(x,q,a)|=0.19\left(	\frac{(x+1)^2}{\log (x+1)^2}
	-\frac{x^2}{\log x^2} \right) =0.19 \left( \frac{2x+1}{2\log x}+O\left( \frac{x}{(\log x)^2}\right)\right) .$$
\end{lem}	

\begin{proof}[\textbf{Proof}]			
	Expanding and simplifying the reciprocal of the first denominator yield
	\begin{eqnarray}\label{eq3377D.200k}
		\frac{1}{\log (x+1)^2}&=&\frac{1}{2[\log x+\log(1+1/x)]}\\[.3cm]
	&=&\frac{1}{2\log x}\cdot \frac{1}{\displaystyle 1+\frac{\log(1+1/x)}{\log x}}\nonumber\\[.3cm]
		&=&\frac{1}{2\log x}\cdot \frac{1}{\displaystyle 1+O\left( \frac{1}{x\log x}\right)}\nonumber\\[.3cm]
		&=&\frac{1}{2\log x}\left( 1+O\left( \frac{1}{x\log x}\right) \right) \nonumber.
	\end{eqnarray}
The second line in \eqref{eq3377D.200k} uses $\log(1+z)=O(z)$ for $|z|<1$
and the third line uses $1/(1+z)=1+O(z)$ for $|z|<1$, to complete the simplification. Replacing the new denominator \eqref{eq3377D.200k} back yields
	\begin{eqnarray}\label{eq3377D.200m}
		\frac{(x+1)^2}{\log (x+1)^2}
		-\frac{x^2}{\log x^2}
		&=&\frac{(x+1)^2}{2\log x}\left( 1+O\left( \frac{1}{x\log x}\right) \right) -\frac{x^2}{2\log x}\\[.3cm]
		&=&\frac{2x+1}{2\log x}+O\left( \frac{x^2+2x+1}{x(\log x)^2}\right) \nonumber\\[.3cm]
		&=&\frac{2x+1}{2\log x}+O\left( \frac{x}{(\log x)^2}\right)\nonumber .
	\end{eqnarray}
\end{proof}

\begin{thm} \label{thm3377U.050}\hypertarget{thm3377U.050} Let $1\leq a<q\leq 10^5$, with $\gcd(a,q)=1$ be a pair of integers and let $n\geq 1$ be a large integer. Then, there exists a prime number $p\equiv a\bmod q$ in the range $n^2< p<(n+1)^2$. Furthermore, for a large number $x$, the number of weighted primes over the short interval $[x^2,(x+1)^2]$ has the upper bound
	$$	\psi((x+1)^2,q,a)-\psi(x^2,q,a)\leq \frac{2x+1}{\varphi(q)}+O\left( \frac{x}{\log x}\right).$$
\end{thm}

\begin{proof}[\textbf{Proof}] The same routines as in the previous result mutatis mutandis.
\end{proof}

Merging the result in \hyperlink{thm3377D.050}{Theorem} \ref{thm3377D.050} and \hyperlink{thm3377U.050}{Theorem} \ref{thm3377U.050} leads to the inequalities
\begin{equation}
\frac{2x}{\varphi(q)}+O\left( \frac{x}{\log x}\right)\leq 	\psi((x+1)^2,q,a)-\psi(x^2,q,a)\leq \frac{2x+1}{\varphi(q)}+O\left( \frac{x}{\log x}\right).
\end{equation}

There are many numerical data on the Legendre conjecture, the most recent appears in \cite{SW2024}. 

\section{Primes Between Two Cubic Integers and Beyond }\label{S3377Y}  

A superficial inspection of the proof for primes between two squares given in \hyperlink{S3377V}{Section} \ref{S3377V} seems to work as well for primes between two cubic integers. Nevertheless, there is another method. The maximal gap $p_{n+1}-p_n\ll x^{5/8}$ between consecutive primes $p_n\leq x$, proved in \cite{IA1937}, implies the existence of primes between cubes, the most recent proof is in \cite{DA2016}, other cases are considered in \cite[Section 5]{CJ2023}. \\

An easy proof based on the maximal gap $p_{n+1}-p_n\ll x^{.525}$ between consecutive primes $p_n\leq x$ proved in \cite{BP2001} is assembled below.

\begin{thm} \label{thm3377Y.300}\hypertarget{thm3377Y.300} Let $1\leq a<q\leq 10^5$, with $\gcd(a,q)=1$ be a pair of integers and let $n\geq 1$ be a large integer. Then, there exists a prime number $p\equiv a\bmod q$ in the range $n^v< p<(n+1)^v$ for all $v\geq3$. 
	
\end{thm}	

\begin{proof}[\textbf{Proof}] By Theorem 1 in \cite{BP2001}, the interval $[x^v,x^v+x^{0.525v}]$ contains primes. Now observe that the binomial expansion of the $(x+1)^v=x^v+(v(v-1)/2)x^{v-1}+\cdots+1$, so there is a proper set inclusion of intervals
\begin{equation}
[x^v,x^v+x^{0.525v}]	\subset[x^v,x^v+cx^{v-1}]	\subset [x^v,(x+1)^v]
\end{equation}	
for any $v\geq3$ such that $0.525v< v-1$, where $c\geq v(v-1)\geq6$ is a small constant. The claim follows from these information.
\end{proof}



\end{document}